\documentclass[12pt]{amsart}

\usepackage{amssymb}
\usepackage[abbrev]{amsrefs}

\newcommand{\N}{\mathbb N}

\providecommand{\norm}[1]{\left\lVert#1\right\rVert}
\providecommand{\ip}[2]{\left\langle #1, #2 \right\rangle}
\newcommand{\F}{\operatorname{F}}
\newcommand{\Fhat}{\hat{\mathrm F}}

\allowdisplaybreaks
\numberwithin{equation}{section}

\theoremstyle{plain}
\newtheorem{theorem}{Theorem}[section]
\newtheorem{lemma}[theorem]{Lemma}
\newtheorem{corollary}[theorem]{Corollary}

\theoremstyle{remark}
\newtheorem{remark}[theorem]{Remark}

\title[Parallel methods]{%
Parallel methods for quasinonexpansive mappings in a Hilbert space}

\author{Koji~Aoyama}
\address[K.~Aoyama]
{Aoyama Mathematical Laboratory, Chiba University, Japan}
\email{aoyama@bm.skr.jp}

\author{Shigeru~Iemoto}
\address[S.~Iemoto]{%
Faculty of Science and Engineering, Chuo University, Japan}
\email{iemoto@tamacc.chuo-u.ac.jp}


\begin{document}

\begin{abstract}
This paper is devoted to the problem of finding a common fixed point
of quasinonexpansive mappings defined on a Hilbert space. 
To approximate the solution to this problem, we present several 
iterative processes using the parallel method 
based on \cites{MR3196179,JMM2018}. 
\end{abstract}

\maketitle

\section{Introduction}
The objective of this study is to solve the following common fixed point
problem: 
\begin{center}
 Find $z \in \bigcap_{i=1}^N \F(T_i)$, 
\end{center}
where $N$ is a positive integer,
$T_i$ is a quasinonexpansive mapping of a subset $C$ of a Hilbert space
$H$, and $\F(T_i)$ is the set of fixed points of $T_i$ for $i \in
\{1,\dotsc, N\}$. 
In particular, we intend to prove that if $u$ is a given point in $H$,
then the sequence $\{ x_n \}$ generated by $u$ and $T_i$ converges
strongly to the common fixed point of $\{T_i\}$ closest to $u$. 

In the previous study \cite{JMM2018}, 
we established some strong convergence results regarding the problem 
by the parallel hybrid method and the parallel shrinking method;
see \cite{MR3196179} and the references therein for the parallel
hybrid method. 
These results were proven using convergence theorems obtained in 
\cite{MR2884574} and some properties \cite{JMM2018}*{Lemmas 2.2 and
2.3} of the parallel algorithm.

In the present study, 
we obtain other convergence results for the common fixed point
problem
by using the parallel method based on \cites{MR3196179,JMM2018} 
and some simple iteration processes
developed in~\cites{MR2529497,MR2960628}. 

The paper is organized as follows: Section 2 introduces some relevant
notions, definitions, and lemmas. 
In Section 3, we prove a 
strong convergence theorem for the common fixed
point problem of quasinonexpansive nonself-mappings. 
In Section 4, we also provide strong and weak convergence theorems 
for the common fixed point problem of strongly quasinonexpansive
mappings. 

\section{Preliminaries}

Throughout the present paper, $H$ denotes a real Hilbert space, 
$\ip{\,\cdot\,}{\,\cdot\,}$ the inner product of $H$, 
$\norm{\,\cdot\,}$ the norm of $H$, $C$ a nonempty subset of $H$, 
$I$ the identity mapping on $H$, and $\N$ the set of positive
integers. 
Strong convergence of a sequence $\{x_n\}$ in $H$ to $z\in H$ is
denoted by $x_n \to z$ and weak convergence by 
$x_n \rightharpoonup z$. 

Let $T\colon C\to H$ be a mapping. 
The set of fixed points of $T$ is denoted by $\F(T)$, that is, 
$\F(T) = \{z \in C\colon Tz=z\}$. 
A point $z \in C$ is said to be an \emph{asymptotic fixed point} of
$T$~\cites{MR2058234} if there exists a sequence $\{x_n\}$ in
$C$ such that $x_n - T x_n\to 0$ and $x_n \rightharpoonup z$. 
The set of asymptotic fixed points of $T$ is denoted by $\Fhat(T)$. 
It is clear that $\F(T) \subset \Fhat(T)$. 
A mapping $T$ is said to be \emph{quasinonexpansive} if 
$\F(T) \ne \emptyset$ and $\norm{Tx - z} \leq \norm{x-z}$ for all
$x\in C$ and $z\in \F(T)$; 
$T$ is said to be \emph{nonexpansive} if 
$\norm{Tx - Ty} \leq \norm{x-y}$ for all $x,y\in C$; 
$T$ is said to be \emph{firmly nonexpansive}~\cite{MR0215141} if 
$\ip{x-y}{Tx - Ty} \geq \norm{Tx - Ty}^2$, or
equivalently
\begin{equation}\label{e:firmly}
\norm{Tx - Ty}^2 \leq \norm{x-y}^2 - \norm{x-y-(Tx-Ty)}^2
\end{equation}
for all $x, y \in C$; 
$T$ is said to be \emph{strongly quasinonexpansive}~\cites{%
pNACA2017,pNACA2015,MR4136436,MR3908455,MR3213161}
if $T$ is quasinonexpansive and $Tx_n - x_n \to 0$ whenever
$\{x_n\}$ is a bounded sequence in $C$ and 
$\norm{x_n - p} - \norm{Tx_n - p} \to 0$ for some point $p\in \F(T)$;
$T$ is said to be \emph{demiclosed at $0$} if $z \in C$ and $Tz = 0$ 
whenever $\{ x_n \}$ is a sequence in $C$ such that
$x_n\rightharpoonup z$ and
$T x_n \to 0$; see, for example, \cite{MR1074005}. 
It is clear that $I-T$ is demiclosed at $0$ if and only if 
$\Fhat(T) = \F(T)$. 
Moreover, under the assumption that $C$ is closed and convex, 
we know the following:
\begin{itemize}
 \item If $T$ is quasinonexpansive, then $\F(T)$
       is closed and convex; see~\cite{MR0298499}*{Theorem~1};
 \item if $T$ is nonexpansive, then $\Fhat(T) = \F(T)$; 
       see~\cite{MR1074005}. 
\end{itemize}

\begin{remark}
 The notion of an asymptotic fixed point above originates from
 that proposed by Reich~\cite{MR1386686}. 
 However, the two notions are slightly different. 
\end{remark}

\begin{remark}
 The notion of a strongly quasinonexpansive mapping is based on that of
 a strongly nonexpansive mapping in the sense of~\cite{MR1386686};
 see also~\cites{MR0470761}. 
\end{remark}

\begin{remark}
 If a strongly nonexpansive mapping in the sense 
 of~\cites{MR0470761,MR2581778,MR2799767,MR2671943,MR2377867}
 has a \emph{fixed point}, then it is strongly quasinonexpansive 
 above. 
\end{remark}

\begin{remark}
 Let $T\colon C\to H$ be a mapping with a fixed point. 
 Since $\beta_n - \alpha_n \to 0 \Leftrightarrow 
 \beta_n^2 - \alpha_n^2 \to 0$ for all bounded sequences
 $\{\alpha_n\}$ and $\{\beta_n\}$ in $[0,\infty)$, 
 $T$ is strongly quasinonexpansive if and only if 
 $T$ is of \emph{type~(sr)} in the sense of 
 \cites{MR3258665,MR2529497,MR2884574}; see also 
 \cites{MR2358984,MR3745220}.
\end{remark}

\begin{remark}
 It is known that 
 the subgradient projection in the sense of~\cite{MR3213161}
 is a strongly quasinonexpansive mapping 
 which is not nonexpansive; see also \cites{MR1895827,MR3745220}.
\end{remark}

Let $D$ be a nonempty closed convex subset of $H$. 
It is known that for each $x \in H$
there exists a unique point $x_0 \in D$ such that 
\[
 \norm{x - x_0} = \min\{\norm{x-y}: y\in D\}. 
\]
Such a point $x_0$ is denoted by $P_D (x)$ and $P_D$ is called the
\emph{metric projection} of $H$ onto $D$. 
It is known that $P_D$ is firmly nonexpansive; see, for example,
\cite{MR2548424}.

\begin{remark}\label{r:metric_projection}
 Let $D$ be a nonempty closed convex subset of $H$. 
 Since a firmly nonexpansive mapping is nonexpansive
 by~\eqref{e:firmly}, it follows that the metric projection $P_D$ of
 $H$ onto $D$ is nonexpansive. Thus, $\Fhat(P_D) = \F(P_D) = D$. 
 Moreover, \eqref{e:firmly} also shows that a firmly nonexpansive
 mapping with a fixed point is strongly quasinonexpansive. 
 Therefore, $P_D$ is strongly quasinonexpansive. 
\end{remark}

Let $\{T_n\}$ be a sequence of mappings of $C$ into $H$ 
such that $\bigcap_n \F(T_n)$ is nonempty. Then 
\begin{itemize}
 \item $\{T_n\}$ is said to be 
       a \emph{strongly quasinonexpansive sequence}
       if each $T_n$ is quasinonexpansive and 
       $T_n x_n - x_n \to 0$ whenever $\{x_n\}$ is a bounded sequence
       in $C$ and $\norm{x_n - p} - \norm{T_n x_n - p} \to 0$ 
       for some point $p\in \bigcap_n \F(T_n)$; 
 \item a point $z \in C$ is said to be an \emph{asymptotic fixed
       point} of $\{T_n\}$ if there exist a sequence $\{ x_n \}$ in
       $C$ and a subsequence $\{x_{n_i}\}$ of $\{x_n\}$ such that
       $T_n x_n - x_n \to 0$ and $x_{n_i} \rightharpoonup z$;
       see~\cites{MR3698250,pNACA2017,MR3013135}. 
\end{itemize}
The set of asymptotic fixed points of $\{T_n\}$ is denoted by
$\hat{\F}(\{T_n\})$. It is clear that 
$\bigcap_n \F(T_n) \subset \hat{\F}(\{T_n\})$.

\begin{remark}
 We know that a strongly quasinonexpansive sequence above is 
 a sequence of \emph{strongly quasinonexpansive type}
 in the sense of~\cites{MR3698250,MR3213161,pNACA2017}
 and a \emph{strongly relatively nonexpansive sequence} in the sense 
 of \cites{MR3258665,MR2960628,MR2529497,MR3013135}; 
 see~\cite{MR3213161}*{Remark~2.5}. 
 We also know that if a \emph{strongly nonexpansive sequence} in the
 sense of~\cites{MR2377867,MR2581778,MR2799767} has a common fixed
 point, then it is a strongly quasinonexpansive sequence above;
 see also \cite{MR3714903}. 
\end{remark}

\begin{remark}
 We know that the condition $\Fhat(\{T_n\}) = \bigcap_n \F(T_n)$
 is equivalent to 
 the \emph{condition~(Z)} in the sense of~\cites{%
 MR3258665,MR3185784,MR3213161,MR2960628,
 MR2799767,MR2671943,MR2529497,
 MR3203624,pNACA2011,pNLMUA2011,MR2762191,MR2762173,MR3013135}, 
 the \emph{condition~(Z1)} in the sense of~\cite{pNAO-Asia2008}, and 
 the \emph{condition~(A)} in the sense of~\cites{MR2581778,MR2884574};
 see \cite{MR3013135}*{Proposition 6}. 
\end{remark}

\begin{remark}
 We can find some examples of strongly quasinonexpansive sequences with the
 condition~(Z) in~\cite{MR3698250}, \cite{MR3213161}*{Example 4.5}, 
 and~\cites{MR2799767,MR2529497,MR2960628}; 
 see also~\cites{MR2762191,MR2762173,MR3203624,pNACA2011,MR2671943,
 MR2780284,pNAO-Asia2008,MR2581778,MR2377867,MR3258665,MR3013135}. 
\end{remark}

The following lemma is a direct consequence
of~\cite{MR2529497}*{Theorem 3.8} and
~\cite{MR3013135}*{Proposition~4}; 
see also~\cite{MR2671943}*{Lemma 2.2}. 

\begin{lemma}\label{l:SQNt-Z}
 Let $H$ be a Hilbert space, $C$ a nonempty subset of $H$, 
 $\{U_n\}$ be a sequence of mappings of $C$ into $H$ such that 
 $\bigcap_n \F(U_n)$ is nonempty, 
 $\{\beta_n\}$ a sequence in $(0,1)$, 
 and $S_n\colon C\to H$ a mapping defined by 
 $S_n = \beta_n I + (1-\beta_n)U_n$ for $n\in\N$. 
 Then $\F(U_n) = \F(S_n)$ for all $n \in \N$, and moreover, 
 the following hold: 
 \begin{itemize}
  \item If each $U_n$ is quasinonexpansive and $\inf_n \beta_n > 0$,
	then 
	$\{S_n\}$ is a strongly quasinonexpansive sequence; 
  \item if $\sup_n \beta_n <1$ and 
	$\Fhat(\{U_n\}) = \bigcap_n \F(U_n)$, then 
	$\Fhat (\{S_n\}) = \bigcap_n \F(S_n)$. 
 \end{itemize}
\end{lemma}

The following lemma is important in the convergence analysis 
of the parallel method in Sections 3 and 4. 

\begin{lemma}\label{l:T_i2U_n}
 Let $H$ be a Hilbert space, $C$ a nonempty subset of $H$, 
 $\{U_n\}$ be a sequence of mappings of $C$ into $H$,
 $N$ a positive integer, 
 and $T_i$ a mapping of $C$ into $H$ for $i \in \Lambda$, 
 where $\Lambda = \{i \in \N\colon 1\leq i \leq N\}$. 
 Suppose that $\bigcap_{i \in \Lambda} \F(T_i)$ is nonempty, 
 and that for any $x \in C$ and $n \in \N$ there exists
 $k \in \arg \max \{\norm{T_i x - x} \colon i \in \Lambda \}$
 such that $U_n x = T_k x$. 
 Then the following hold: 
 \begin{enumerate}
  \item $\F(U_n) = \bigcap_{i \in \Lambda} \F(T_i)$
	for all $n \in \N$, and hence $\bigcap_n \F(U_n)$ is nonempty; 
  \item if $T_i$ is quasinonexpansive for all $i \in \Lambda$, 
	then so is $U_n$ for all $n \in \N$;
  \item if $\Fhat(T_i) = \F(T_i)$ for all $i \in \Lambda$, 
	then $\Fhat(\{U_n\}) = \bigcap_n \F(U_n)$;
  \item if $T_i$ is a strongly quasinonexpansive mapping
	for all $i \in \Lambda$, then $\{U_n\}$ is a strongly
	quasinonexpansive sequence. 
 \end{enumerate}
\end{lemma}

\begin{proof}
 Because (1), (2), and (3) are direct consequences
 of~\cite{JMM2018}*{Lemmas 2.2 and 2.3}, we must only prove (4).
 From (1) and (2), we know that 
 $\bigcap_n \F(U_n) = \bigcap_{i \in \Lambda} \F(T_i)$ is nonempty
 and each $U_n$ is quasinonexpansive. 
 Let $\{x_n\}$ be a bounded sequence in $C$ and suppose that 
 $\norm{x_n - p} - \norm{U_n x_n - p} \to 0$ for some 
 $p \in \bigcap_n \F(U_n)$. 
 We denote $U_n x_n - x_n$ by $z_n$. 
 It is enough to show that $z_n \to 0$. 
 Let $\{z_{n_m}\}$ be a subsequence of $\{ z_n \}$ and 
 $\{U_{n_m}x_{n_m}\}$ a subsequence of $\{ U_n x_n\}$ with the same index.
 By assumption, it turns out that 
 there exist $i \in \Lambda$ and a subsequence 
 $\{U_{n_{m_l}} x_{n_{m_l}}\}$ of $\{U_{n_m} x_{n_m}\}$ 
 such that $U_{n_{m_l}} x_{n_{m_l}} = T_i x_{n_{m_l}}$ for all 
 $l \in \N$. 
 Since $p \in \F(T_i)$, 
 $T_i$ is strongly quasinonexpansive, and
 \[
 \norm{x_{n_{m_l}} - p} - \norm{T_i x_{n_{m_l}} - p} 
 = \norm{x_{n_{m_l}} - p} - \norm{U_{n_{m_l}} x_{n_{m_l}} - p} 
 \to 0
 \]
 as $l \to \infty$, 
 it follows that $z_{n_{m_l}} = T_i x_{n_{m_l}} - x_{n_{m_l}} \to 0$. 
 This completes the proof. 
\end{proof}

The following theorem is a direct consequence 
of~\cite{MR2960628}*{Theorem 4.1}; see 
also~\cites{pNACA2017,pNACA2011,MR3213161,MR3698250}. 

\begin{theorem}\label{t:AKT2012JNAO}
 Let $H$ be a Hilbert space, $C$ a nonempty closed convex subset of $H$, 
 $u$ a point in $H$, $\{\alpha_n\}$ a sequence in $(0,1]$, 
 $\{S_n\}$ a sequence of mappings of $C$ into $H$, 
 and $\{x_n\}$ a sequence defined by $x_1 \in C$ and
 \[
 x_{n+1} = P_C \bigl(\alpha_n u + (1-\alpha_n) S_n x_n\bigr)
 \]
 for $n \in \N$. 
 Suppose that 
 $F = \bigcap_n \F(S_n) \ne \emptyset$,
 $\{S_n\}$ is a strongly quasinonexpansive sequence,
 $\hat{\F}(\{S_n\}) = F$, 
 $\sum_n \alpha_n =\infty$, and $\lim_n\alpha_n=0$.
 Then $\{x_n\}$ converges strongly to $P_F (u)$. 
\end{theorem}

The following theorem is a direct consequence 
of~\cite{MR2529497}*{Theorem 4.1}; see also~\cite{MR2377867}. 

\begin{theorem}\label{t:AKT2009JFPTA}
 Let $H$ be a Hilbert space, $C$ a nonempty closed convex subset of $H$, 
 $\{S_n\}$ a sequence of mappings of $C$ into itself, 
 and $\{x_n\}$ a sequence in $C$ defined by $x_1 \in C$ and
 $x_{n+1} = S_n x_n$ for $n \in \N$. 
 Suppose that $F = \bigcap_n \F(S_n) \ne \emptyset$,
 $\{S_n\}$ is a strongly quasinonexpansive sequence, and
 $\hat{\F}(\{S_n\}) = F$. 
 Then $\{x_n\}$ converges weakly to the strong limit of $\{P_F(x_n)\}$. 
\end{theorem}

\section{Parallel methods for quasinonexpansive mappings}

In this section, using the parallel method and 
Theorem~\ref{t:AKT2012JNAO},
we prove a strong convergence theorem for the common fixed point
problem of quasinonexpansive nonself-mappings. 

\begin{theorem} \label{t:Strong-QN}
 Let $H$ be a Hilbert space, 
 $C$ a nonempty closed convex subset of $H$, $N$ a positive integer, 
 and $T_i \colon C\to H$ a quasinonexpansive mapping for $i \in
 \Lambda$, where $\Lambda = \{ i \in \N\colon 1 \leq i \leq N\}$. 
 Let $F$ be the set of common fixed points of
 $\{T_i \}_{i\in \Lambda}$, that is,
 $F = \bigcap_{i \in \Lambda} \F(T_i)$, 
 and $\{x_n\}$ a sequence defined by $x_1\in C$ and 
 \[
 \begin{cases}
  i_n \in \arg \max \{ \norm{T_i x_n - x_n} \colon i \in \Lambda \}; \\
  x_{n+1} = P_C \Bigl( \alpha_n u + (1-\alpha_n) 
  \bigl( \beta_n x_n + (1-\beta_n) T_{i_n} x_n \bigr) \Bigr)
 \end{cases}
 \]
 for $n \in \N$, where $u \in H$, $\{\alpha_n\}$ is a sequence in
 $(0,1]$, and $\{\beta_n\}$ is a sequence in $(0,1)$. 
 Suppose that $F$ is nonempty, 
 $\Fhat(T_i) = \F(T_i)$ for all $i \in \Lambda$, 
 $\sum_n \alpha_n = \infty$, $\lim_n \alpha_n =0$, 
 $\inf_n \beta_n >0$, and $\sup_n \beta_n <1$. 
 Then $\{x_n\}$ converges strongly to $P_{F} (u)$. 
\end{theorem}

\begin{proof}
 Let $M$ be a set-valued mapping of $C$ into $\Lambda$ defined by 
 \[
 M(x) = \arg \max \{ \norm{T_i x - x} \colon i \in \Lambda\} 
 \]
 for $x\in C$ and let $U_n\colon C \to H$ be a mapping defined by 
 \[
  U_n x =
   \begin{cases}
    T_{i_n} x & \text{if $x=x_n$;}\\
    T_{j(x)} x & \text{otherwise}
   \end{cases}
 \]
 for $x \in C$ and $n \in \N$, where $j(x) = \min M(x)$ for $x \in C$. 
 Then we can verify that for any $x \in C$ and $n \in \N$, there exists
 $k \in M(x)$ such that $U_n x = T_k x$. 
 Thus Lemma~\ref{l:T_i2U_n} implies that 
 $F= \bigcap_n \F(U_n) \ne \emptyset$, 
 each $U_n$ is quasinonexpansive, 
 and $\Fhat(\{U_n\}) = \bigcap_n \F(U_n)$. 
 Set $S_n = \beta_n I + (1-\beta_n)U_n$ for $n \in \N$. 
 Then Lemma~\ref{l:SQNt-Z} shows that 
 $\{S_n\}$ is a strongly quasinonexpansive sequence
 and $\Fhat(\{S_n\}) = \bigcap_n \F(S_n) = F$. 
 From the definition of $U_n$, it is clear that 
 $x_{n+1} = P_C \bigl(\alpha_n u + (1-\alpha_n) S_n x_n\bigr)$ for all
 $n \in \N$. 
 Using Theorem~\ref{t:AKT2012JNAO}, 
 we conclude that $\{x_n\}$ converges strongly to $P_{F} (u)$. 
\end{proof}

If we suppose that every $T_i$ is a self-mapping and $u\in C$ in
Theorem~\ref{t:Strong-QN}, then we deduce the following: 

\begin{corollary}
 Let $H$, $C$, $N$, $\Lambda$, $\{\alpha_n\}$, and $\{\beta_n\}$ 
 be the same as in Theorem~\ref{t:Strong-QN}. 
 Let $T_i \colon C\to C$ be a quasinonexpansive mapping for
 $i \in \Lambda$, $F$ the set of common fixed points of 
 $\{T_i \}_{i\in \Lambda}$, and 
 $\{x_n\}$ a sequence defined by $x_1\in C$ and 
 \[
 \begin{cases}
  i_n \in \arg \max \{ \norm{T_i x_n - x_n} \colon i \in \Lambda \}; \\
  x_{n+1} = \alpha_n u + (1-\alpha_n) 
  \bigl( \beta_n x_n + (1-\beta_n) T_{i_n} x_n \bigr)
 \end{cases}
\]
 for $n \in \N$, where $u \in C$. 
 Suppose that $F$ is nonempty
 and $\Fhat(T_i) = \F(T_i)$ for all $i \in \Lambda$. 
 Then $\{x_n\}$ converges strongly to $P_{F} (u)$. 
\end{corollary}

\section{Parallel methods for strongly quasinonexpansive mappings}
In this section, we prove strong and weak convergence theorems for the
common fixed point problem of strongly quasinonexpansive
self-mappings. 
We then explain how the theorems can be applied to common fixed point
problems for strongly quasinonexpansive nonself-mappings 
and quasinonexpansive mappings. 

Using Lemma~\ref{l:T_i2U_n} and Theorem~\ref{t:AKT2012JNAO}, 
we obtain the following: 

\begin{theorem}\label{t:strong-SQN}
 Let $H$ be a Hilbert space, 
 $N$ a positive integer, 
 $S_i \colon H\to H$ a strongly quasinonexpansive mapping for
 $i \in \Lambda$, and $F$ the set of common fixed points of
 $\{S_i \}_{i\in \Lambda}$, 
 where $\Lambda = \{ i \in \N\colon 1 \leq i \leq N\}$. 
 Suppose that $F$ is nonempty
 and $\Fhat(S_i) = \F(S_i)$ for all $i \in \Lambda$. 
 Let $\{x_n\}$ be a sequence defined by $x_1\in H$ and 
 \[
 \begin{cases}
  i_n \in \arg\max \{ \norm{S_i x_n - x_n} \colon i \in \Lambda \};\\
  x_{n+1} = \alpha_n u + (1-\alpha_n) S_{i_n} x_n 
 \end{cases}
 \]
 for $n \in \N$, where $u\in H$
 and $\{\alpha_n\}$ is a sequence in
 $(0,1]$. 
 If $\sum_n \alpha_n = \infty$ and $\lim_n \alpha_n =0$, 
 then $\{x_n\}$ converges strongly to $P_{F} (u)$. 
\end{theorem}

\begin{proof}
 Let $M$ be a set-valued mapping of $H$ into $\Lambda$ defined by 
 \begin{equation}\label{e:Mx:SQN}
  M(x) = \arg \max \{ \norm{S_i x - x} \colon i \in \Lambda\}
 \end{equation}
 for $x\in H$ and let $U_n\colon H \to H$ be a mapping defined by 
 \begin{equation}\label{e:U_n:SQN}
  U_n x =
   \begin{cases}
    S_{i_n} x & \text{if $x=x_n$;}\\
    S_{j(x)} x & \text{otherwise}
   \end{cases}
 \end{equation}
 for $x \in H$ and $n \in \N$, where $j(x) = \min M(x)$ for $x \in H$. 
 Then we can verify that for any $x \in H$ and $n \in \N$, there exists
 $k \in M(x)$ such that $U_n x = S_k x$. 
 Thus Lemma~\ref{l:T_i2U_n} implies that 
 $F= \bigcap_n \F(U_n) \ne \emptyset$, 
 $\{U_n\}$ is a strongly quasinonexpansive sequence, 
 and $\Fhat(\{U_n\}) = \bigcap_n \F(U_n)$. 
 From the definition of $U_n$, it is clear that 
 $x_{n+1} = \alpha_n u + (1-\alpha_n) U_n x_n$ for all $n \in \N$. 
 Using Theorem~\ref{t:AKT2012JNAO}, 
 we conclude that $\{x_n\}$ converges strongly to $P_{F} (u)$. 
\end{proof}

Using Lemma~\ref{l:T_i2U_n} and Theorem~\ref{t:AKT2009JFPTA}, 
we obtain the following weak convergence result: 

\begin{theorem}\label{t:weak-SQN}
 Let $H$, $N$, $S_i$, $\Lambda$, $F$ be the same as in 
 Theorem~\ref{t:strong-SQN} 
 and $\{x_n\}$ a sequence defined by $x_1\in H$ and 
 \[
 \begin{cases}
  i_n \in \arg \max \{ \norm{S_i x_n - x_n} \colon i \in \Lambda \}; \\
  x_{n+1} = S_{i_n} x_n 
 \end{cases}
 \]
 for $n \in \N$. 
 Then $\{x_n\}$ converges weakly to the strong limit of $\{P_F(x_n)\}$. 
\end{theorem}

\begin{proof}
 Let $M$ be a set-valued mapping of $H$ into $\Lambda$ defined by 
 \eqref{e:Mx:SQN} for $x\in H$
 and $U_n\colon H \to H$ a mapping defined by 
 \eqref{e:U_n:SQN} for $x \in H$ and $n \in \N$. 
 As in the proof of Theorem~\ref{t:strong-SQN}, 
 it follows from Lemma~\ref{l:T_i2U_n} that 
 $F = \bigcap_n \F(U_n) \ne \emptyset$, 
 $\{U_n\}$ is a strongly quasinonexpansive sequence, 
 and $\Fhat(\{U_n\}) = \bigcap_n \F(U_n)$. 
 From the definition of $U_n$, 
 it is clear that $x_{n+1} = U_n x_n$ for all $n \in \N$. 
 Applying Theorem~\ref{t:AKT2009JFPTA}, we complete the proof. 
\end{proof}

\begin{remark}
 Theorems~\ref{t:strong-SQN} and \ref{t:weak-SQN} are applicable to
 the construction of a common fixed point of a finite family of strongly
 quasinonexpansive \emph{nonself}-mappings. 
 Indeed, let $C$ be a nonempty closed convex subset of a Hilbert space
 $H$ and $S\colon C\to H$ a strongly quasinonexpansive mapping. 
 Then it is clear that $SP_C$ is a mapping of $H$ into $H$. Furthermore,
 it is known that 
 $\F(SP_C) = \F(S)$ and $SP_C$ is strongly quasinonexpansive; 
 see~\cite{MR2884574}*{Lemmas 3.2 and 3.3} and
 ~\cite{MR3908455}*{Theorem 5.5}. 
 Moreover, it is known that 
 $\Fhat(SP_C) = \F(SP_C)$ if $\Fhat(S) = \F(S)$; 
 see \cite{MR2884574}*{Theorem 3.4} 
 and Remark~\ref{r:metric_projection}.
\end{remark}

\begin{remark}
 Theorems~\ref{t:strong-SQN} and \ref{t:weak-SQN} are applicable to
 the construction of a common fixed point of a finite family of
 \emph{quasinonexpansive} mappings. 
 Indeed, let $T\colon H\to H$ be a quasinonexpansive mapping. 
 Set $S_\lambda = \lambda I + (1-\lambda)T$ for some 
 $\lambda \in (0,1)$. 
 Then it is known that $\F(S_\lambda) = \F(T)$ and 
 $S_\lambda$ is strongly quasinonexpansive; 
 see \cite{MR2358984}*{Corollary 3.4}. 
 Moreover, it is known that 
 $\Fhat(S_\lambda) = \F(S_\lambda)$ if $\Fhat(T) = \F(T)$; 
 see~\cite{MR2884574}*{Corollary 3.8}. 
\end{remark}



\begin{bibdiv}
\begin{biblist}

\bib{MR3196179}{article}{
      author={Anh, Pham~Ky},
      author={Chung, Cao~Van},
       title={Parallel hybrid methods for a finite family of relatively
  nonexpansive mappings},
        date={2014},
        ISSN={0163-0563},
     journal={Numer. Funct. Anal. Optim.},
      volume={35},
       pages={649\ndash 664},
         url={http://dx.doi.org/10.1080/01630563.2013.830127},
}

\bib{MR2762191}{incollection}{
      author={Aoyama, Koji},
       title={An iterative method for a variational inequality problem over the
  common fixed point set of nonexpansive mappings},
        date={2010},
   booktitle={Nonlinear analysis and convex analysis},
   publisher={Yokohama Publ., Yokohama},
       pages={21\ndash 28},
}

\bib{MR2762173}{incollection}{
      author={Aoyama, Koji},
       title={An iterative method for fixed point problems for sequences of
  nonexpansive mappings},
        date={2010},
   booktitle={Fixed point theory and its applications},
   publisher={Yokohama Publ., Yokohama},
       pages={1\ndash 7},
}

\bib{MR3013135}{incollection}{
      author={Aoyama, Koji},
       title={Asymptotic fixed points of sequences of quasi-nonexpansive type
  mappings},
        date={2011},
   booktitle={Banach and function spaces {III} ({ISBFS} 2009)},
   publisher={Yokohama Publ., Yokohama},
       pages={343\ndash 350},
}

\bib{pNLMUA2011}{incollection}{
      author={Aoyama, Koji},
       title={Halpern's iteration for a sequence of quasinonexpansive type
  mappings},
        date={2011},
   booktitle={Nonlinear mathematics for uncertainty and its applications},
      editor={Li, Shoumei},
      editor={Wang, Xia},
      editor={Okazaki, Yoshiaki},
      editor={Kawabe, Jun},
      editor={Murofushi, Toshiaki},
      editor={Guan, Li},
   publisher={Springer Berlin Heidelberg},
     address={Berlin, Heidelberg},
       pages={387\ndash 394},
         url={http://dx.doi.org/10.1007/978-3-642-22833-9_47},
}

\bib{pNACA2011}{incollection}{
      author={Aoyama, Koji},
       title={Approximations to solutions of the variational inequality problem
  for inverse-strongly-monotone mappings},
        date={2013},
   booktitle={Nonlinear analysis and convex analysis -{I}-},
   publisher={Yokohama Publ., Yokohama},
       pages={1\ndash 9},
}

\bib{pNACA2015}{incollection}{
      author={Aoyama, Koji},
       title={Strongly quasinonexpansive mappings},
        date={2016},
   booktitle={Nonlinear analysis and convex analysis},
   publisher={Yokohama Publ., Yokohama},
       pages={19\ndash 27},
}

\bib{MR3698250}{article}{
      author={Aoyama, Koji},
       title={Viscosity approximation method for quasinonexpansive mappings
  with contraction-like mappings},
        date={2016},
        ISSN={1341-9951},
     journal={Nihonkai Math. J.},
      volume={27},
       pages={167\ndash 180},
         url={https://projecteuclid.org/euclid.nihmj/1505419750},
}

\bib{MR3714903}{article}{
      author={Aoyama, Koji},
       title={Uniformly nonexpansive sequences},
        date={2017},
        ISSN={2188-8159},
     journal={Linear Nonlinear Anal.},
      volume={3},
       pages={179\ndash 187},
}

\bib{JMM2018}{article}{
      author={Aoyama, Koji},
       title={Parallel hybrid methods for relatively nonexpansive mappings},
        date={2018},
        ISSN={1344-7777},
     journal={Josai Mathematical Monographs},
      volume={11},
       pages={121\ndash 130},
}

\bib{pNACA2017}{incollection}{
      author={Aoyama, Koji},
       title={Approximation of fixed points of nonexpansive mappings and
  quasinonexpansive mappings in a {H}ilbert space},
        date={2019},
   booktitle={Nonlinear analysis and convex analysis},
   publisher={Yokohama Publ., Yokohama},
       pages={1\ndash 10},
}

\bib{MR2799767}{article}{
      author={Aoyama, Koji},
      author={Kimura, Yasunori},
       title={Strong convergence theorems for strongly nonexpansive sequences},
        date={2011},
        ISSN={0096-3003},
     journal={Appl. Math. Comput.},
      volume={217},
       pages={7537\ndash 7545},
         url={http://dx.doi.org/10.1016/j.amc.2011.01.092},
}

\bib{MR3203624}{incollection}{
      author={Aoyama, Koji},
      author={Kimura, Yasunori},
       title={A note on the hybrid steepest descent methods},
        date={2013},
   booktitle={Fixed point theory and its applications},
   publisher={Casa C\u ar\c tii de \c Stiin\c t\u a, Cluj-Napoca},
       pages={73\ndash 80},
}

\bib{MR3185784}{article}{
      author={Aoyama, Koji},
      author={Kimura, Yasunori},
       title={Viscosity approximation methods with a sequence of contractions},
        date={2014},
        ISSN={0716-7776},
     journal={Cubo},
      volume={16},
       pages={9\ndash 20},
}

\bib{MR2960628}{article}{
      author={Aoyama, Koji},
      author={Kimura, Yasunori},
      author={Kohsaka, Fumiaki},
       title={Strong convergence theorems for strongly relatively nonexpansive
  sequences and applications},
        date={2012},
        ISSN={1906-9685},
     journal={J. Nonlinear Anal. Optim.},
      volume={3},
       pages={67\ndash 77},
}

\bib{MR2377867}{article}{
      author={Aoyama, Koji},
      author={Kimura, Yasunori},
      author={Takahashi, Wataru},
      author={Toyoda, Masashi},
       title={On a strongly nonexpansive sequence in {H}ilbert spaces},
        date={2007},
        ISSN={1345-4773},
     journal={J. Nonlinear Convex Anal.},
      volume={8},
       pages={471\ndash 489},
}

\bib{MR2581778}{incollection}{
      author={Aoyama, Koji},
      author={Kimura, Yasunori},
      author={Takahashi, Wataru},
      author={Toyoda, Masashi},
       title={Strongly nonexpansive sequences and their applications in
  {B}anach spaces},
        date={2008},
   booktitle={Fixed point theory and its applications},
   publisher={Yokohama Publ., Yokohama},
       pages={1\ndash 18},
}

\bib{MR3258665}{article}{
      author={Aoyama, Koji},
      author={Kohsaka, Fumiaki},
       title={Strongly relatively nonexpansive sequences generated by firmly
  nonexpansive-like mappings},
        date={2014},
        ISSN={1687-1812},
     journal={Fixed Point Theory Appl.},
       pages={2014:95, 13},
         url={http://dx.doi.org/10.1186/1687-1812-2014-95},
}

\bib{MR3213161}{article}{
      author={Aoyama, Koji},
      author={Kohsaka, Fumiaki},
       title={Viscosity approximation process for a sequence of
  quasinonexpansive mappings},
        date={2014},
        ISSN={1687-1812},
     journal={Fixed Point Theory Appl.},
       pages={2014:17, 11},
         url={http://dx.doi.org/10.1186/1687-1812-2014-17},
}

\bib{MR3745220}{article}{
      author={Aoyama, Koji},
      author={Kohsaka, Fumiaki},
       title={Cutter mappings and subgradient projections in {B}anach spaces},
        date={2017},
        ISSN={2188-8159},
     journal={Linear Nonlinear Anal.},
      volume={3},
       pages={457\ndash 473},
}

\bib{MR4136436}{article}{
      author={Aoyama, Koji},
      author={Kohsaka, Fumiaki},
       title={Strongly quasinonexpansive mappings, {III}},
        date={2020},
        ISSN={2188-8159},
     journal={Linear Nonlinear Anal.},
      volume={6},
       pages={1\ndash 12},
}

\bib{MR2671943}{article}{
      author={Aoyama, Koji},
      author={Kohsaka, Fumiaki},
      author={Takahashi, Wataru},
       title={Shrinking projection methods for firmly nonexpansive mappings},
        date={2009},
        ISSN={0362-546X},
     journal={Nonlinear Anal.},
      volume={71},
       pages={e1626\ndash e1632},
         url={http://dx.doi.org/10.1016/j.na.2009.02.001},
}

\bib{MR2884574}{incollection}{
      author={Aoyama, Koji},
      author={Kohsaka, Fumiaki},
      author={Takahashi, Wataru},
       title={Strong convergence theorems by shrinking and hybrid projection
  methods for relatively nonexpansive mappings in {B}anach spaces},
        date={2009},
   booktitle={Nonlinear analysis and convex analysis},
   publisher={Yokohama Publ., Yokohama},
       pages={7\ndash 26},
}

\bib{pNAO-Asia2008}{incollection}{
      author={Aoyama, Koji},
      author={Kohsaka, Fumiaki},
      author={Takahashi, Wataru},
       title={Strong convergence theorems for a family of mappings of type
  ({P}) and applications},
        date={2009},
   booktitle={Nonlinear analysis and optimization},
   publisher={Yokohama Publ., Yokohama},
       pages={1\ndash 17},
}

\bib{MR2529497}{article}{
      author={Aoyama, Koji},
      author={Kohsaka, Fumiaki},
      author={Takahashi, Wataru},
       title={Strongly relatively nonexpansive sequences in {B}anach spaces and
  applications},
        date={2009},
        ISSN={1661-7738},
     journal={J. Fixed Point Theory Appl.},
      volume={5},
       pages={201\ndash 224},
         url={http://dx.doi.org/10.1007/s11784-009-0108-7},
}

\bib{MR2780284}{article}{
      author={Aoyama, Koji},
      author={Kohsaka, Fumiaki},
      author={Takahashi, Wataru},
       title={Proximal point methods for monotone operators in {B}anach
  spaces},
        date={2011},
        ISSN={1027-5487},
     journal={Taiwanese J. Math.},
      volume={15},
       pages={259\ndash 281},
}

\bib{MR2358984}{article}{
      author={Aoyama, Koji},
      author={Takahashi, Wataru},
       title={Strong convergence theorems for a family of relatively
  nonexpansive mappings in {B}anach spaces},
        date={2007},
        ISSN={1583-5022},
     journal={Fixed Point Theory},
      volume={8},
       pages={143\ndash 160},
}

\bib{MR3908455}{article}{
      author={Aoyama, Koji},
      author={Zembayashi, Kei},
       title={Strongly quasinonexpansive mappings, {II}},
        date={2018},
        ISSN={1345-4773},
     journal={J. Nonlinear Convex Anal.},
      volume={19},
       pages={1655\ndash 1663},
}

\bib{MR1895827}{article}{
      author={Bauschke, Heinz~H.},
      author={Combettes, Patrick~L.},
       title={A weak-to-strong convergence principle for {F}ej\'er-monotone
  methods in {H}ilbert spaces},
        date={2001},
        ISSN={0364-765X},
     journal={Math. Oper. Res.},
      volume={26},
       pages={248\ndash 264},
         url={http://dx.doi.org/10.1287/moor.26.2.248.10558},
}

\bib{MR0215141}{article}{
      author={Browder, Felix~E.},
       title={Convergence theorems for sequences of nonlinear operators in
  {B}anach spaces},
        date={1967},
        ISSN={0025-5874},
     journal={Math. Z.},
      volume={100},
       pages={201\ndash 225},
         url={https://doi.org/10.1007/BF01109805},
}

\bib{MR0470761}{article}{
      author={Bruck, Ronald~E.},
      author={Reich, Simeon},
       title={Nonexpansive projections and resolvents of accretive operators in
  {B}anach spaces},
        date={1977},
        ISSN={0362-1588},
     journal={Houston J. Math.},
      volume={3},
       pages={459\ndash 470},
}

\bib{MR0298499}{article}{
      author={Dotson, W.~G., Jr.},
       title={Fixed points of quasi-nonexpansive mappings},
        date={1972},
        ISSN={0263-6115},
     journal={J. Austral. Math. Soc.},
      volume={13},
       pages={167\ndash 170},
}

\bib{MR1074005}{book}{
      author={Goebel, Kazimierz},
      author={Kirk, W.~A.},
       title={Topics in metric fixed point theory},
      series={Cambridge Studies in Advanced Mathematics},
   publisher={Cambridge University Press, Cambridge},
        date={1990},
      volume={28},
        ISBN={0-521-38289-0},
         url={http://dx.doi.org/10.1017/CBO9780511526152},
}

\bib{MR2058234}{article}{
      author={Matsushita, Shin-ya},
      author={Takahashi, Wataru},
       title={Weak and strong convergence theorems for relatively nonexpansive
  mappings in {B}anach spaces},
        date={2004},
        ISSN={1687-1820},
     journal={Fixed Point Theory Appl.},
       pages={37\ndash 47},
         url={http://dx.doi.org/10.1155/S1687182004310089},
}

\bib{MR1386686}{incollection}{
      author={Reich, Simeon},
       title={A weak convergence theorem for the alternating method with
  {B}regman distances},
        date={1996},
   booktitle={Theory and applications of nonlinear operators of accretive and
  monotone type},
      series={Lecture Notes in Pure and Appl. Math.},
      volume={178},
   publisher={Dekker, New York},
       pages={313\ndash 318},
}

\bib{MR2548424}{book}{
      author={Takahashi, Wataru},
       title={Introduction to nonlinear and convex analysis},
   publisher={Yokohama Publishers, Yokohama},
        date={2009},
        ISBN={978-4-946552-35-9},
}

\end{biblist}
\end{bibdiv}
\end{document}